\documentclass[atmp]{ipart_v1}

\Vol{19}
\Issue{3}
\Year{2015}
\firstpage{701}

\usepackage{t1enc}
\usepackage[latin1]{inputenc}
\usepackage[english]{babel}

\usepackage{amsthm}
\usepackage{yfonts}

\usepackage{bbm}
\usepackage{bm}
\usepackage{mathrsfs}

\newcommand{\be}[0]{\begin{equation}}
\newcommand{\ee}[0]{\end{equation}}

\numberwithin{equation}{section}

\theoremstyle{plain}
\newtheorem{theorem}{Theorem}[section]

\begin{document}

\title[Parabolic equations]{A singularly perturbed system of parabolic equations}

\author[Asan Omuraliev, ~ Peiil Esengul kyzy]{Asan Omuraliev, ~ Peiil Esengul kyzy}


\begin{abstract}
The work is devoted to the construction of the asymptotic behavior of the solution of a singularly perturbed system of equations of parabolic type, in the case when the limit equation has a regular singularity as the small parameter tends to zero. The asymptotics of the solution of such problems contains boundary layer functions.
\end{abstract}

\maketitle

\section{Introduction and summary}

We consider the first boundary-value problem for a system of singularly perturbed parabolic equations
\begin{equation}\label{eq1}
L_\varepsilon u\equiv (\varepsilon+t)\partial_t u-\varepsilon^2 A(x)\partial_x^2 u-D(t)u=f(x,t), ~~ (x,t)\in\Omega 
\end{equation}
$$
u(x,0,\varepsilon)=h(x), ~~ u(0,t,\varepsilon)=u(1,t,\varepsilon)=0,
$$
where $ \Omega=(0,1)\times(0,T], ~~ \varepsilon>0$ is small parameter, $u=u(x,t,\varepsilon)= \\ col(u_1(x,t,\varepsilon),u_2(x,t,\varepsilon),...,u_n(x,t,\varepsilon)) $,  $ A(x)\in C^\infty([0,1],\mathbb C^{n^2})$, \\ $ D(t)\in C^\infty([0,T],\mathbb C^{n^2})$,  $ f(x,t)\in C^\infty(\bar{\Omega},\mathbb C^n)$.

The work is a continuation of [1], where instead of the matrix - function $A(x)$ there was a scalar function and an asymptotic of the solution was constructed containing two functions describing the boundary layers along $x=0$ and $x=1$. In this case, the asymptotics contains 2m parabolic boundary layer functions describing the boundary layers along  $x=0$ and $x=1$. 

Construction of the asymptotic solution of a singularly perturbed system of parabolic equations is devoted  works [2] - [5]. In [2], a regularized asymptotic  is constructed in the case when the matrix of coefficients for the desired function has zero multiple eigenvalue. A similar problem was studied in [3] and an asymptotic of the boundary layer type was constructed. The method of boundary functions in [4] studied the bisingular problem for systems of parabolic equations, which is characterized by the presence of nonsmoothness of the asymptotic terms and a singular dependence on a small parameter. In [5] - [6], various problems for split systems of two equations of parabolic type were studied, and asymptotics of the boundary layer type were constructed. The problems of differential equations of parabolic type with a small parameter were studied in [7]-[9]. In [10]-[13], numerical methods for solving singularly perturbed problems are studied.

\section{Statement of the problem}

We consider the first boundary-value problem (\ref{eq1}).

The problem is solved under the following assumptions:

1) For $n$ is dimensional vector-valued functions $f(x,t)$ and $h(x)$ performed inclusions
$$
f(x,t)\in C^\infty (\bar{\Omega},\mathbb C^n), ~~~ h(x)\in C^\infty ([0,1],\mathbb C^n),
$$
for $n\times n$ are matrix-valued functions $D(t)$ and $A(x)$  inclusion 
$$
D(t)\in C^\infty ([0,T],\mathbb C^{n\times n}), ~~~ A(x)\in C^\infty ([0,1],\mathbb C^{n\times n});
$$

2) The real parts of all roots $\lambda_i(x)\lambda, ~~ i=\overline{1,n},$ of the equation $det(A(x)-\lambda E)=0$ are positive and $\lambda_i(x)\neq\lambda_j(x)$ for all $x\in[0,1],$  if $i\neq j, ~~~ i,j=\overline{1,n}$;

3) The real parts of the eigenvalues $\beta_j(t), ~~ j=\overline{1,n}$ of the matrix $D(t)$ are not positive, i.e. $Re(\beta_j(t))\leq 0$ and $\beta_i(0)\neq\beta_j(t) ~~ \forall t \in [0,T], ~~ i\neq j, ~~~ i,j=\overline{1,n}$;

4) Completed the conditions of approval  the initial and boundary conditions $h(0)=h(1)=0$.

\section{Regularization of the problem}

Following [1, p.316; 5, p.18], we introduce regularizing variables

\begin{equation}\label{eq3}
\xi_{i,l}=\frac{\varphi_{i,l}(x)}{\sqrt{\varepsilon^3}};  \varphi_{i,l}(x)=(-1)^{l-1}\int_{l-1}^x \frac{ds}{\sqrt{\lambda_i(s)}}, l=1,2, i=\overline{1,n}, \end{equation}
$$
\tau=\frac{1}{\varepsilon}ln(\frac{t+\varepsilon}{\varepsilon}), ~~ \mu_j=\beta_j(0)ln(\frac{t+\varepsilon}{\varepsilon})\equiv K_j(t,\varepsilon), ~ j=\overline{1,n}, 
$$
and an extended function such that
$$
\tilde{u}(M,\varepsilon)|_{\xi=\varphi(x)/\varepsilon}\equiv u(x,t,\varepsilon), ~ M=(x,t,\xi,\tau,\mu), ~~ \xi=(\xi_1,\xi_2), \\
$$
\begin{equation}\label{eq4}
~~ \xi_l=(\xi_{1,l},\xi_{2,l},...,\xi_{n,l}), \varphi(x)=(\varphi_1(x),\varphi_2(x)), ~~  \end{equation}
$$
\varphi_l(x)=(\varphi_{1,l},\varphi_{2,l}(x),...,\varphi_{n,l}(x)), ~~ l=1,2, ~~ \mu=(\mu_1,\mu_2,...,\mu_n).
$$

Based on (\ref{eq3}), we find the derivatives from (\ref{eq4}):

$$
\partial_t u\equiv\left(\partial_t\tilde{u}+\frac{1}{\varepsilon(t+\varepsilon)}\partial_\tau\tilde{u}+\sum_{j=1}^n \frac{\beta_j(0)}{t+\varepsilon}\partial_{\mu_j}\tilde{u}\right)|_{\theta=\chi(x,t,\varepsilon)},
$$
$$
\partial_x u\equiv\left(\partial_x \tilde{u}+\sum_{l=1}^2 \sum_{i=1}^n\left[\frac{1}{\sqrt{\varepsilon^3}}\acute{\varphi_{i,l}}(x)\partial_{\xi_i,l}\tilde{u}\right]\right)|_{\xi=\varphi(x)/\sqrt{\varepsilon^3}},
$$
$$
\partial_x^2 u\equiv\left(\partial_x^2\tilde{u}+\sum_{l=1}^2 \sum_{i=1}^n \left[\frac{1}{\varepsilon^3}(\acute{\varphi}_{i,l}(x))^2\partial_{\xi_{i,l}}^2\tilde{u}+\frac{1}{\sqrt{\varepsilon^3}}L_{i,l}^\xi\tilde{u}\right]\right)|_{\xi=\varphi(x)/\sqrt{\varepsilon^3}},
$$
$$
L_{\tau,l}^\xi\tilde{u}=2\varphi^{'}_{i,l}(x)\partial_{x\xi_{i,l}}^2 \tilde{u}+\varphi^{''}_{i,l}(x)\partial_{\xi_{i,l}}\tilde{u}, ~~~ \chi(x,t,\varepsilon)=\left(\frac{\varphi(x)}{\sqrt{\varepsilon^3}},\frac{1}{\varepsilon}ln(\frac{t+\varepsilon}{\varepsilon}),K(t,\varepsilon)\right),
$$$$
\varphi=(\varphi_1, \varphi_2), ~~ \varphi_l=(\varphi_{1,l},\varphi_{2,l},...,\varphi_{n,l})
$$
$$
K(t,\varepsilon)=\left(K_1(t,\varepsilon),K_2(t,\varepsilon),...,K_n(t,\varepsilon)\right), ~~~ \theta=(\tau,\mu).
$$

According to these calculations, as well as (\ref{eq2}) and (\ref{eq4}), we pose the following extended problem:
\begin{equation}
\label{eq5}
\tilde{L}_\varepsilon \tilde{u}\equiv\varepsilon\partial_t\tilde{u}+\frac{1}{\varepsilon}T_0\tilde{u}+T_1\tilde{u}-\sqrt{\varepsilon}L^\xi\tilde{u}-\varepsilon^2L_x\tilde{u}=f(x,t), \end{equation}
$$
\tilde{u}|_{t=0}=h(x), ~~ \tilde{u}|_{x=0,\xi_{i,1}=0}=\tilde{u}|_{x=1,\xi_{i,2}=0}=0, ~~ i=\overline{1,n}, 
$$
$$T_0\tilde{u}=\partial_\tau\tilde{u}-\sum_{l=1}^2\sum_{i=1}^n \left(\varphi^{'}_{i,l}(x)\right)^2 A(x)\partial_{\xi_{i,l}}^2\tilde{u}, ~~ T_1\tilde{u}=t\partial_t\tilde{u}+\sum_{j=1}^\infty \beta_j(0)\partial_{\mu_j}\tilde{u}-D(t)\tilde{u},
$$$$
L^\xi\tilde{u}=\sum_{l=1}^2\sum_{i=1}^n A(x)L_{i,l}^\xi\tilde{u}, ~~~ L_x\tilde{u}=A(x)\partial_x^2\tilde{u}.
$$

In this case, the identity holds:
\begin{equation}\label{eq6}
\left(\tilde{L}_\varepsilon\tilde{u}(M,\varepsilon)\right)_{\theta=\chi(x,t,\varepsilon)}\equiv L_\varepsilon u(x,t,\varepsilon).
\end{equation}
The solution to the extended problem (\ref{eq5}) will be defined as a series:
\begin{equation}\label{eq7}
\tilde{u}(M,\varepsilon)=\sum_{k=0}^\infty \varepsilon^{k/2}u_k(M).
\end{equation}

Substituting (\ref{eq7}) into problem (\ref{eq5}) and equating the coefficients for the same powers of P, we obtain the following equations:
\begin{equation}\label{eq8}
T_0 u_0=0, ~~ T_0 u_1=0, ~~ T_0 u_2=f(x,t)-T_1 u_0, ~~ T_0 u_3=L_\xi u_0-T_1 u_1,
\end{equation}
$$
T_0 u_k = L_\xi u_{k-3}+L_x u_{k-6}-\partial_t u_{k-4}-T_1 u_{k-2}. 
$$

The initial and boundary conditions for them are set in the form:
\begin{equation}\label{eq9}
u_0|_{t=\tau=\mu=0}=h(x), ~~ u_k(M)|_{t=\mu=\tau=0}=0, \end{equation}
$$
u_k(M)|_{x=l-1,\xi_{i,l}=0}=0, ~~ k\geq 1, ~~ i=\overline{1,n}, ~~ l=1,2. $$

\section{Solvability of iterative problems}

Each of problems (\ref{eq8}) has innumerable solutions; therefore, we single out a class of functions in which these problems were uniquely solvable. We introduce the following function classes:
$$
U_1=\left\{V(x,t): V(x,t)=\sum_{i=1}^n v_i(x,t)\psi_i(t), ~~ v_i(x,t)\in C^\infty(\bar{\Omega})\right\},
$$$$
U_2=\left\{Y(N): Y(N)=\sum_{l=1}^2 \sum_{i=1}^n y_i^l(N_i^l)b_i(x), ~~ \left|y_i^l(N_i^l)\right|<c\exp\left(-\frac{\xi_{i,l}}{8\tau}\right)\right\},
$$$$
U_3=\left\{C(x,t): C(x,t)=\sum_{i=1}^n\left[\sum_{j=1}^n c_{ij}(x,t)\exp(\mu_j)+p_i(x)\right]\psi_i(t), ~~ c_{ij}(x,t)\in C^\infty (\bar{\Omega})\right\},
$$$$
U_4=\left\{Z(N): Z(N)=\sum_{l=1}^2 \sum_{i,j=1}^n z_{i,j}^l(N_i^l)b_i(x)\exp(\mu_j), ~~ \left|z_{i,j}^l(N_i^l)\right|<c\exp\left(-\frac{\xi_{i,l}^2}{8\tau}\right)\right\},
$$
where $N_i^l=(x,t,\xi_{i,l},\tau,\mu_i), ~~ i=1,2,...,n, ~~ l=1,2.$  From these classes of functions we construct a new class as a direct sum:
$$
U=U_1\oplus U_2\oplus U_3\oplus U_4.
$$
The function $u_k(M)\in U$ is representable in vector form:
$$
u_k(M)=\Psi(t)\left[V_k(x,t)+C^k(x,t)\exp(\mu)\right]+
$$$$
+\sum_{l=1}^2 B(x)\left[Y^{k,l}(N^l)+Z^{k,l}(N^l)\exp(\mu)\right], ~~ C^k(x,t)=C^k_1(x,t)+\Lambda(P(x)),
$$$$
C^k_1(x,t)=(c_{ij}(x,t)), ~~ V_k(x,t)=col(v_{k1},v_{k2},...,v_{kn}), 
$$$$ Y^{k,l}(N^l)=col(Y^{k,l}_1(N^l),Y^{k,l}_2(N^l),...,Y^{k,l}_n(N^l)), ~~ Z^{k,l}(N^l)=(z^{k,l}_{ij}(N^l)), 
$$$$ \Psi(t)=(\psi_1(t),\psi_2(t),...,\psi_n(t)), ~~ B(x)=(b_1(x),b_2(x),...,b_n(x)),
$$
$$
\exp(\mu)=col(\exp(\mu_1),\exp(\mu_2),...,\exp(\mu_n)).
$$
or in coordinate form:
\begin{equation}
\label{eq10}
u_k(M)=\sum_{i=1}^n v_{k,i}(x,t)\psi_i(t)+\sum_{l=1}^2 \sum_{i=1}^n y_i^{k,l}(N_i^l)b_i(x)+ \end{equation}
$$
+\sum_{i,j=1}^n \left\{c_{i,j}^k(x,t)\psi_i(t)+\sum_{l=1}^2 z_{i,j}^{k,l}(N_i^l)b_i(x)\right\}\exp(\mu_j)+\sum_{i=1}^n p_i^k(x)\psi_i(t)\exp(\mu_i). $$

The vector functions $b_i(x), ~ \psi_i(t)$ included in these classes are eigenfunctions of the matrices $A(x)$ and $D(t)$, respectively:
\begin{equation}\label{eq11}
A(x)b_i(x)=\lambda_i(x)b_i(x), ~~ D(t)\psi_i(t)=\beta_i(t)\psi_i(t), ~~ i=\overline{1,n}.
\end{equation}
Moreover, according to condition 1), they are smooth of their arguments.

Along with the eigenvectors $b_i(x)$ and $\psi_i(t)$, the eigenvectors $b_i^*(x), \psi_i^*(t), ~~ i = \overline{1,n}$ of the conjugate matrices $A^*(x), ~~ D^*(t)$ will be used:
$$
A^*(x)b_i^*(x)=\bar{\lambda}_i(x)b^*_i(x), ~~~ D^*(t)\psi_i^*(t)=\bar{\beta}_i(t)\psi_i^*(t)
$$
and they are selected biorthogonal:
$$
\left(b_i(x),b^*_j(x)\right)=\delta_{i,j}, ~~~ \left(\psi_i(t),\psi_j^*(t)\right)=\delta_{i,j}, ~~ i,j=\overline{1,n}
$$
where $\delta_{i,j}$ is Kronecker symbol.

We calculate the action of the operators $T_0, T_1, L_\xi, L_x$ on the function $u_k(M,\varepsilon) $ from (\ref{eq10}), taking into account relations (\ref{eq11}) and
$$
\varphi^{'2}_{i,l}(x)=\frac{1}{\lambda_i(x)}, ~ i=\overline{1,n}
$$
we have:
\begin{equation}\label{eq12}
T_0 u_k(M)\equiv \sum_{i=1}^n \sum_{l=1}^2 \left\{\partial_\tau y_i^{k,l}(N_i^l)-\partial_{\xi_{i,l}}^2 y_i^{k,l}(N_i^l)+  \right.\end{equation}
$$
+\sum_{j=1}^n \left.\left[\partial_\tau z_{i,j}^{k,l}(N_i^l)-\partial_{\xi_{i,l}}^2 z_{i,j}^{k,l}(N_i^l)\right]\exp(\mu_j)\right\}b_i(x);  $$

or vector form
$$
T_0 u_k(M)\equiv \sum_{l=1}^2 B(x)\left\{\partial_\tau Y^{k,l}(N^l)-\partial_{\xi_l}^2 Y^{k,l}(N^l)+\right.$$
$$\left.+\left[\partial_\tau Z^{k,l}(N^l)-\partial_{\xi_l}^2 Z^{k,l}(N^l)\right]\exp(\mu)\right\}
$$
$Y^{k,l}(N^l)$ is $n$- vector, $Z^{k,l}(N^l)$ is $n\times n$ - matrix.

Here $B(x)$ is a matrix function $(n\times n)$ whose columns are the eigenvectors $b_i(x)$ of the matrix $A(x)$.
We calculate
$$
T_1 u_k=t\partial_t u_k+\sum_{j=1}^n \beta_j(0)\partial_{\mu_j} u_k-D(t)u_k=
$$$$
t\sum_{i=1}^n \left\{\left[\partial_t v_{k,i}+\sum_{r=1}^n \alpha_{r,i}(t)v_{k,r}(x,t)\right]\psi_i(t)+\sum_{l=1}^2 \partial_t y_i^{k,l}(N^l_i)b_i(x)+\right.
$$
\begin{equation}\label{eq13}
+\sum_{j=1}^n\left[\left(\partial_t c_{ij}^k(x,t)+\sum_{r=1}^n \alpha_{ri}(t)c_{r,j}^k(x,t)+\alpha_{ji}(t)p_j^k(x)\right)\psi_i(t)+\right.
\end{equation}
$$\left.\left.+\sum_{l=1}^2 z_{ij}^{k,l}(N_i^l)b_i(x)\right]\exp(\mu_j)\right\}+
$$$$+\sum_{i,j=1}^n \beta_j(0)c_{ij}^k(x,t)\psi_i(t)\exp(\mu_j)+\sum_{i=1}^n \beta_i(0)p_i^k(x)\psi_i(t)\exp(\mu_j)+
$$$$
+\sum_{l=1}^2 \sum_{i,j=1}^n \beta_j(0)z_{i,j}^{k,l}(x,t)b_i(x)\exp(\mu_j)-
$$$$
-\sum_{i=1}^n \left\{\beta_i(t)v_{ki}(x,t)\psi_i(t)+\sum_{l=1}^2 \sum_{r=1}^n \gamma_{r,i}(x,t
)y_i^{k,l}(N_i^l)b_i(x)\right\}-
$$$$
-\sum_{i,j=1}^n \beta_j(t)c_{i,j}^k(x,t)\psi_i(t)\exp(\mu_j)-\sum_{i=1}^n \beta_i(t)p_i^k(x)\psi_i(t)\exp(\mu_i)-
$$$$
-\sum_{l=1}^2 \sum_{i,j=1}^n \sum_{r=1}^n \gamma_{r,i}(x,t)z_{r,j}^{k,l}(N_i^l)b_i(x)\exp(\mu_j)
$$
here $\alpha_{i,r}=\left(\psi^{'}_i (t),\psi^{*}_r(t)\right), ~~ \gamma_{i,r}(x,t)=\left(D(t)b_i(x),b^*_r(x)\right).$

It will be shown below that the scalar functions $y_i^{k,l}(N^l)$ and $z_{i,j}^{k,l}(N^l)$ are representable in the form:
$$
y_i^{k,l}(N^l)=d_i^{k,l}(x,t)y_i^{k,l}(\xi_l,\tau), ~~ z_{i,j}^{k,l}(N^l)=\omega_{i,j}^{k,l}(x,t)z_i^{k,l}(\xi_l,\tau).
$$
Given these representations, we calculate:
\begin{equation} \label{eq14} 
L_\xi u_k(M)=\sum_{l=1}^2 A(x)\sum_{i=1}^n \left\{\left[2\varphi^{'}_i(x)\left(b_i(x)d_i^{k,l}(x,t)\right)^{'}_x + \right.\right. \end{equation}
$$
\left.+\varphi^{''}_i(x)\left(b_i(x)d_i^{k,l}(x,t)\right)\right]\partial_{\xi_{i,l}}y_i^{k,l}(\xi_{i,l},\tau)+  $$$$
+\sum_{j=1}^n \left[2\varphi^{'}_i(x)\left(b_i(x)\omega_{i,j}^{k,l}(x,t)\right)^{'}_x + \right. \\   $$$$
\left.\left.+\varphi^{''}_i(x)\left(b_i(x)\omega_{i,j}^{k,l}(x,t)\right)\right]\partial_{\xi_{i,l}}z_{i,j}^{k,l}(\xi_{i,l},\tau)\exp(\mu_j)\right\}, $$

$$
L_x u_k(M)=A(x)\left[\partial_x^2 (V_k(x,t))+\sum_{l=1}^2 \partial_x^2 \left(B(x)Y^{l,k}(N^l)\right)+\right. 
$$$$
\left.+\partial_x^2 \left(\Psi(t)C^{k,0}(x,t)\right)\exp(\mu)+\sum_{l=1}^2 \partial_x^2 \left(B(x)Z^{l,k}(N^l)\right)\exp(\mu)\right].
$$
Satisfy the function (\ref{eq10}) of the boundary conditions from (\ref{eq5})
\begin{equation} \label{eq15}
y_i^{k,l}(N_i^l)|_{t=\tau=\mu=0}=0, ~~ z_{i,j}^{k,l}(N_i^l)|_{t=\tau=\mu=0}=0, \end{equation}
$$
c_{ii}^k(x,0)=-v_{k,i}(x,0)-p_i^k(x)-\sum_{i\neq j} c_{ij}^k(x,0), $$$$ 
y_i^{k,l}(N_i^l)|_{\xi_{i,l}=0}=d_i^{k,l}(x,t), ~~ d_i^{k,l}(x,t)b_i(x)|_{x=l-1}=-v_i(l-1,t)\psi_i(t),  $$$$
z_{i,j}^{k,l}(N_i^l)|_{\xi_{i,l}=0}=\omega_{i,j}^{k,l}(x,t), $$$$
\omega_{i,j}^{k,l}(x,t)b_i(x)|_{x=l-1}=-\left[c_{i,j}(l-1,t)+p_i(l-1)\right]\psi_i(t).  $$

In general form, the iterative equations (\ref{eq8}) are written as
\begin{equation}\label{eq16}
T_0 u_k(M)=h_k(M).
\end{equation}
\begin{theorem}
Let $h_k(M)\in U$ and conditions 2), 3) hold. Then equation (\ref{eq16}) has a solution $u_k(M)\in U$ if the equations are solvable
\begin{equation}
\label{eq17}
\partial_\tau y_i^{k,l}(N_i^l)-\partial_{\xi_{i,l}}^2 y_i^{k,l}(N_i^l)=h_i^{k,l}(N^l)\equiv \bar{h_i}^{k,1}(x,t)\bar{\bar{h_i}}^{k,1}(\xi_{i,1},\tau)   \end{equation}
$$
\partial_\tau z_{i,j}^{k,l}(N_i^l)-\partial_{\xi_{i,l}}^2 z_{i,j}^{k,l}(N_i^l)=h_{ij}^{k,2}(N^l)\equiv \bar{h}_{ij}^{k,2}(x,t)\bar{\bar{h_i}}^{k,2}(\xi_{i,2},\tau).  
$$

\end{theorem}
\begin{proof}
Let $h_k(M)=\sum_{l=1}^2 \sum_{i=1}^n \left[h^{k,1}_i(N^l)+\sum_{j=1}^n h_{ij}^{k,2}(N^l)\exp(\mu_j)\right]b_i(x)\in U$. Pose (\ref{eq10}) into equation (\ref{eq16}), then, on the basis of  calculations (\ref{eq14}), with respect to $y_i^{k,l}(N_i^l), ~~ z_{ij}^{k,l}(N_i^l)$, we obtain equations (\ref{eq17}).

These equations, under appropriate boundary value conditions:
$$
y_i^{k,l}(N_i^l)|_{t=\tau=\mu=0}=0, ~~ y_i^{k,l}(N_i^l)|_{\xi_{i,l}=0}=d_i^{k,l}(x,t)
$$
$$
z_{i,j}^{k,l}(N_i^l)|_{t=\tau=\mu=0}=0, ~~ y_{i,j}^{k,l}(N_i^l)|_{\xi_{i,l}=0}=\omega_{i,j}^{k,l}(x,t)
$$
Solutions are represented in the form
\begin{equation}\label{eq18}
y_i^{k,l}(N_i^l)=d_i^{k,l}(x,t)erfc\left(\frac{\xi_{i,l}}{2\sqrt{\tau}}\right)+\bar{h_i}^{k,1}(x,t)I_1(\xi_{i,l},\tau),  \end{equation}
$$ 
z_{i,j}^{k,l}(N_i^l)=\omega_{i,j}^{k,l}(x,t)erfc\left(\frac{\xi_{i,l}}{2\sqrt{\tau}}\right)+\bar{h_{ij}}^{k,2}(x,t)I_2(\xi_{i,l},\tau)  $$

$$
I_r(\xi_{i,l},\tau)=\frac{1}{2\sqrt{\pi}}\int_0^\tau \int_0^\infty \frac{\bar{\bar{h}}_i^{k,r}(\eta,s)}{\sqrt{\tau-s}}\left[\exp\left(-\frac{(\xi_{i,l}-\eta)^2}{4(\tau-s)}\right)-\right.
$$$$\left.-\exp\left(-\frac{(\xi_{i,l}+\eta)^2}{4(\tau-s)}\right)\right]d\eta ds, ~~ r=1,2,
$$
where $\bar{h}_i^{k,r}(x,t), ~ \bar{h}_i^{k,r}(\eta,s)$  are known functions.

Evaluation of the integral
$$
\left|I_r(\xi_l)\right|\leq c\exp(-\frac{\xi_i,l^2}{8\tau}).
$$
\end{proof}
\begin{theorem}
Let conditions 1) -4) be satisfied, then equation (\ref{eq8}) under additional conditions
\begin{enumerate}
\item
 $u_k|_{t=\mu=\tau=0}=0, ~~ u_k|_{x=l-1,\xi_{i,l}=0}=0, ~~ l=1,2;$
\item
 $-T_1 u_k-\partial_t u_{k-2}+L_x u_{k-4} \in U_2\oplus U_4;$
\item
 $L_\xi u_k=0$,
\end{enumerate}

has the unique solution in $U$.
\end{theorem}
\begin{proof}
Satisfying the function $u_k(M)\in U$ with the boundary conditions from (\ref{eq2}) we obtain (\ref{eq15}). Based on calculations (\ref{eq12}), (\ref{eq13}), (\ref{eq14}) we have
$$
-T_1 u_k-\partial_t u_{k-2}+L_x u_{k-4}=-t\sum_{i=1}^n \left\{\partial_t V_{k,i}(x,t)\psi_i(t)+\right.
$$$$
+\sum_{r=1}^n \alpha_{ri}(t)V_{k,r}(x,t)\psi_i(t)+\sum_{r=1}^n \alpha_{ri}(t)p^k_r(x)\psi_i(t)\exp(\mu)+$$$$
+\sum_{j=1}^n \left[\partial_t c_{ij}^k+\sum_{r=1}^n \alpha_{r,i}(t)c_{rj}^k(x,t)\right]\psi_i(t)\exp(\mu_j)+
$$$$
\left.+\sum_{l=1}^2 \left[\partial_t y_i^{k,l}(N^l)+\sum_{j=1}^n \partial_t z_{ij}^k(N^l)\exp(\mu_j)\right]b_i(x)\right\}-
$$$$
-\sum_{i,j=1}^n \psi_i(t)c_{ij}^k(x,t)\left(\beta_j(0)-\beta_j(t)\right)\exp(\mu_j)-
$$$$
-\sum_{i=1}^n \psi_i(t)p_i^k(x)\left(\beta_j(0)-\beta_j(t)\right)\exp(\mu_j)+
$$$$
+\sum_{l=1}^2 \sum_{i=1}^n \left\{\sum_{r=1}^n \gamma_{ir}(x,t)b_r(x)y_i^{k,l}(N^l)+\sum_{j=1}^n \sum_{r=1}^n \gamma_{i,r}(x,t)b_r(x)z_{ij}^{k,l}(N^l)\exp(\mu_j)\right\}-
$$$$
-\sum_{i=1}^n \sum_{r=1}^n \alpha_{r,i}(t)p_r^{k-2}(x)\exp(\mu_r)\psi_i(t)-
$$$$
-\sum_{i=1}^n \left\{\partial_t v_{k-2}+\sum_{r=1}^n \alpha_{r,i}(t)v_r+\sum_{j=1}^n \left[\partial_t c_{ij}^{k-2}+\sum_{r=1}^n \alpha_{r,i}(t)c_{r,j}^{k-2}(x,t)\right]\exp(\mu_j)\right\}\psi_i(t)-
$$$$
-\sum_{l=1}^2 \sum_{i=1}^n \left\{\partial_t y_i^{k,l}(N^l)+\sum_{j=1}^n \partial_t z_{ij}^{k,l}(N^l)\exp(\mu_j)\right\}b_i(x)+
$$$$
+A(x)\sum_{i=1}^n \left\{\partial_x^2 V_{k-4,i}+\sum_{j=1}^n \partial_x^2 \left[c_{ij}^{k-4}(x,t)+p_i^{k-4}(x)\right]\exp(\mu_j)\right\}\psi_i(t)+
$$$$
+A(x)\sum_{l=1}^2 \sum_{i=1}^n \partial_x^2 \left(\left[y_i^{k,l}(N^l)+\sum_{j=1}^n z_{ij}^{k,l}(N^l)\exp(\mu_j)\right]b_i(x)\right),
$$

$\alpha_{ir}=\left(\psi_i^{'}(t),\psi_r(t)\right), ~~~ \gamma_{ir}(x,t)=\left(D(t)b_i(x),b_r^*(x)\right)$.

Providing the condition of Theorem 1, we set
\begin{equation}\label{*}
t\left[\partial_t V_k+A^T(t)V_k\right]=-\partial_t V_{k-2}-L_x V_{k-4}(x,t), 
\end{equation}
$$
t\left[\partial_t C^k+A^T (t)\left(C^k(x,t)+\Lambda(P^k(x))\right)\right]+\left[C^k(x,t)+\Lambda(P^k(x))\right]\Lambda(\beta(0))-
$$
\begin{equation}\label{**}
-\Lambda(\beta(t))\left[C^k(x,t)+\Lambda(P^k(x))\right]=    
\end{equation}
$$
=-\partial_t C^{k-2}-A^T(t)\left[C^{k-2}+\Lambda(P^{k-2}(x))\right]+L_x \left[C^{k-4}+\Lambda(P^{k-4}(x))\right],
$$
Equation (\ref{*}) is solved without an initial condition and uniquely determines the function $V_k(x,t)$ [Vazov]. Providing the solvability of equation (\ref{**}) we set
$$
\overline{\overline{C^k(x,t)\Lambda(\beta(0))-\Lambda(\beta(t))C^k(x,t)}}|_{t=0}=
$$$$=-\overline{\overline{\partial_t C^{k-2}-A^T(t)C^{k-2}(x,t)-A^T(t)\Lambda(P^{k-2})+L_x(C^{k-4})}}|_{t=0}
$$$$
\overline{A^T(t)\Lambda(P^{k-2}(x))}=\overline{\left[-\partial_t C^{k-2}-A^T(t)C^{k-2}-L_x\left(C^{k-4}(x,t)+\Lambda(P^{k-4}(x))\right)\right]}
$$
or in coordinate form
$$
c_{ij}^k(x,t)|_{t=0}=-\frac{1}{\beta_j(0)-\beta_i(t)}\left[\partial_t c_{ij}^{k-2}+\sum_{r\neq j}\alpha_{ri}(t)c_{rj}^{k-2}(x,t)-q_{ij}(x,t)\right]_{t=0},
$$$$
p_i^{k-2}(x)=-\frac{1}{\alpha_{ii}(t)}\left[\partial_t c_{ii}^{k-2}+\alpha_{ii}(t)c_{ii}^{k-2}-q_{ii}(x,t)\right]_{t=0},
$$
where $q_{ij}(x,t)$ is known functions included in $A(x)\partial_x^2\left[C^{k-4}(x,t)+\Lambda(P^{k-4}(x))\right]$.

On the basis of (\ref{eq14}), condition 3), Theorem 2 is ensured if arbitrary functions $d_i^{l,k}(x,t)b_i(x), ~~ \omega_{ij}^{k,l}(x,t)b_i(x)$ are solutions to the problems:
$$
2\varphi_{i,l}^{'}(x)\left(d_i^{l,k}(x,t)b_i(x)\right)_x^{'}+\varphi_{i,l}^{''}(x)\left(d_i^{l,k}(x,t)b_i(x)\right)=0
$$
$$
d_i^{l,k}(x,t)b_i(x)|_{x=l-1}=-v_k(l-1,t)\psi_i(t), ~~ i=\overline{1,n}
$$
\begin{equation}\label{eq19}
2\varphi_{i,l}^{'}(x)\left(\omega_{i,j}^{l,k}(x,t)b_i(x)\right)_x^{'}+\varphi_{i,l}^{''}(x)\left(\omega_{i,j}^{l,k}(x,t)b_i(x)\right)=0
\end{equation}
$$
\omega_{i,j}^{l,k}(x,t)b_i(x)|_{x=l-1}=-\left[c_{ij}^k(l-1,t)+p_i(l-1)\right]\psi_j(t).
$$
Thus, arbitrary functions $d_i^{l,k}(x,t), ~~ \omega_{ij}^{k,l}(x,t), ~~ v_{ki}(x,t), ~~ c_{ij}^{k,l}(x,t)$ included in (\ref{eq10}) are uniquely determined.

\end{proof}

\section{Solving iterative problems}
Iterative equation (\ref{eq8}) for $k=0,1$ is homogeneous; therefore, by Theorem 1, it has a solution $u_k(M)\in U$ if the functions $y_i^{l,k}(N_i^l), ~ z_{i,j}^{l,k}(N_i^l)$ are solutions of the equations.
$$
\partial_\tau y_i^{k,l}(N_i^l)=\partial_{\xi_{i,l}^2} y_i^{k,l}(N_i^l),
$$$$
\partial_\tau z_{i,j}^{k,l}(N_i^l)=\partial_{\xi_{i,l}^2} z_{i,j}^{k,l}(N_i^l).
$$
for boundary value conditions
$$
y_i^{k,l}(N_i^l)|_{t=\tau=0}=0, ~~ y_i^{k,l}(N_i^l)|_{\xi_{i,l}=0}=d_i^{k,l}(x,t),
$$
\begin{equation}\label{eq20}
z_{i,j}^{k,l}(N_i^l)|_{t=\tau=0}=0, ~~ z_{i,j}^{k,l}(N_i^l)|_{\xi_{i,l}=0}=\omega_{i,j}^{k,l}(x,t),
\end{equation}
From this problem we find
$$
y_i^{0,l}(N_i^l)=d_i^{0,l}(x,t)erfc\left(\frac{\xi_{i,l}}{2\sqrt{\tau}}\right),
$$$$
z_{i,j}^{0,l}(N_i^l)=\omega_{i,j}^{0,l}(x,t)erfc\left(\frac{\xi_{i,l}}{2\sqrt{\tau}}\right).
$$
The functions $d_i^{0,l}(x,t), ~ \omega_{i,j}^{l,0}(x,t)$ are determined from problems (\ref{eq19}), which ensure that condition $L_\xi u_0=0$. is satisfied. Using calculations (\ref{eq13}), the free term of iterative equation (\ref{eq8}) at $k=2$ is written as
$$
F_2(M)=-T_1 u_0(M)+f(x,t)$$
by Theorem 1, an equation with such a free term is solvable in $U$ if 
$$
t\sum_{i=1}^n \left\{\left[\partial_t v_{0,i}(x,t)+\sum_{r=1}^n \alpha_{ri}(x)v_{0,r}(x,t)\right]-\beta_i(t)v_{0,i}(x,t)\right\}=f(x,t)
$$
$$
t\sum_{i,j=1}^n \left\{\left[\partial_t c_{ij}^0(x,t)+\sum_{r=1}^n \alpha_{ri}(t)c_{rj}^0(x,t)\right]+\alpha_{ji}(t)p_j^0(x)\right\}+
$$
\begin{equation}\label{eq21}
+\sum_{i,j=1}^n \left[\beta_j(0)-\beta_i(t)\right]c_{ij}^0(x,t)+\sum_{i=1}^n \left[\beta_j(0)-\beta_i(t)\right]p_i^0(x)=0.
\end{equation}
From (\ref{eq21}) we uniquely determine $c_{i,j}^0(x,t)=0, ~ \forall i\neq j$ and the function $c_{i,i}^0(x,t)$ is determined from the equation
$$
t\left[\partial_t c_{ii}^0(x,t)+\alpha_{ii}(t)c_{ii}^0(x,t)\right]+\left(\beta_i(0)-\beta_i(t)\right)c_{ii}^0(x,t)+
$$$$
+\left[\beta_i(0)-\beta_i(t)\right]p_i^0(x)=0
$$
under the initial condition 
$$
c_{ii}^0(x,0)=-v_{0,i}(x,0)-p_i^0(x).
$$
The first equation from (\ref{eq21}), by virtue of condition 2), has a solution satisfying the condition [Vazov]
$$
\left\|v^0(x,0)\right\|< \infty.
$$
We calculate the free term of equation (\ref{eq8}) at $k=3$
$$
F_3(M)=-T_1 u_1,
$$
which has the same view as $T_1 u_0$. Providing the solvability of equation $T_0 u_3=-T_1 u_1$ in $U$, with respect to $c_{ij}^1(x,t), ~ v_{1i}(x,t)$ we obtain equations (\ref{eq21}).

In the next step $(k=4)$, the free term has the view
$$
F_4(M)=-T_1 u_2-\partial_t u_0+L_\xi u_1.
$$
The functions $d_i^{1,l}(x,t), ~ \omega_{i,j}^{1,l}(x,t)$ entering the $u_1(M)$ provide the condition $L_\xi u_1=0$. Providing the solvability of the iterative equation at $k=4$, we set
$$
t\left[\partial_t v_{2i}(x,t)+\sum_{k=1}^n \alpha_{ki}(x)v_{2k}(x,t)\right]-\beta_i(t)v_{2i}(x,t)=-\partial_t v_{0i}(x,t).
$$
For $c_{ij}^2(x,t)$, we obtain the same equation of the form (\ref{eq21}), but with the right-hand side $\partial_t c_{ij}^0(x,t)+\sum_{k=1}^n \alpha_{ki}(x)\left(c_{kj}^0(x,t)+\alpha_{ji}(x)p_j^0(x)\right)$. 

Taking off the degeneracy of this equation as $t=0$, we set $p_i^0(x)=-\frac{1}{\alpha_{ii}(t)}\left[\partial_t c_{ii}^0+\alpha_{ii}(x)c_{ii}^0\right]_{t=0}$.

Further repeating the described process, using Theorems 1 and 2, sequentially determining $u_k(M), ~~ k=0,1,...,n$, we construct a partial sum
\begin{equation}
\label{eq22}
u_{\varepsilon n}(M)=\sum_{k=0}^n \varepsilon^{k/2}u_k(M).
\end{equation}

\section{Estimate for the remainder}

For the remainder term 
$$
R_{\varepsilon n}(M)=u(M,\varepsilon)-u_{\varepsilon n }(M)=u(M,\varepsilon)-\sum_{k=0}^{n+1} \varepsilon^{k/2}u_k(M)+\varepsilon^{\frac{n+1}{2}}u_{n+1}
$$
 we get the problem
$$
\tilde{L}_\varepsilon R_{\varepsilon n}(M)=\varepsilon^{\frac{n+1}{2}}g_{\varepsilon n}(M),
$$
$$
R_{\varepsilon n}(M)|_{t=0}=R_{\varepsilon n}|_{x=l-1}=0, ~~ l=1,2.
$$
By narrowing in this problem by means of regularizing functions. Following [Hartman], passing to Euclidean norms, we obtain a problem that is limited by the maximum principle
\begin{equation}
\label{eq23}
\left\|R_{\varepsilon n}(x,t,\varepsilon)\right\|<c \varepsilon^{\frac{n+1}{2}}.
\end{equation}
\begin{theorem}
Let condition 1) -4) be satisfied. Then the restriction of the constructed solution (\ref{eq22}) is an asymptotic solution to the problem (\ref{eq2}), i.e. $\forall n=0,1,...$ the estimate (\ref{eq23}) holds at $\varepsilon \rightarrow 0$.

\end{theorem}

\providecommand{\href}[2]{#2}

\address{
Kyrgyz-Turkish Manas University\\
Chyngyz Aytmatov Campus, 720038, Djal, Bishkek, KYRGYZSTAN\\
\email{asan.omuraliev@manas.edu.kg}, \\
\email{peyil.esengul@manas.edu.kg}\\
}


\begin{thebibliography}{10}

\bibitem{Written:1964xe}
S.A.Lomov, {\em  About the Lithill model equation}.  {MO SSSR, M.}, {\bf
  54}, 74-83, (1964).

\bibitem{Written:2011xe}
S.A.Lomov, I.S.Lomov {\em Fundamentals of the mathematical theory of the boundary layer}.  {MGU, M.}, {\bf}, (2011).

\bibitem{Written:1964xe}
S.A.Lomov, {\em  A power-law boundary layer in problems with a singular perturbation}.  {AN SSSR, Ser-Math.}, {\bf 30:3}, 525-572, (1966).

\bibitem{Written:1964xe}
U.A. Koniev, {\em Construction of an exact solution of some singularly perturbed problems for linear ordinary differential equations with a power-law boundary layer}. {Math. Zametki}, {\bf 79:6}, 950-954, (2006).

\bibitem{Written:1964xe}
U.A. Koniev, {\em Singularly perturbed problems with a double singularity}. {Math. Zametki}, {\bf 62:4}, 494-501, (1997).

\bibitem{Written:1964xe}
A.S. Omuraliev,  {\em Asymptotics of the solution of a system of linear equations of parabolic type with a small parameter}. {Diff. Equation},  {\bf 55:6}, 878-882, (2019). 

\bibitem{Written:1964xe}
	A. Jha and M.K. Kadalbajoo, {\em A robust layer adapted difference method for singularly perturbed two-parameter parabolic problems}. {Int. J. Comput. Math.}, {\bf 92}, 1204-1221, (2015).

\bibitem{Written:1964xe}
	S.D. Shih Prof., {\em On a Class of Singularly Perturbed Parabolic Equations}.  {ZAMM - Journal of Applied Mathematics and Mechanics / Zeitschrift fur Angewandte Mathematik and Mechanik}, {\bf 81, Issue 5}, 2001.

\bibitem{Written:1964xe}
L. Govindarao , J. Mohapatra, {\em  A second-order numerical method for a singularly perturbed delayed partial differential parabolic equation}. {ENGINEERING COMPUTATIONS}, {\bf 36:2}, 420-444, (2019).

\bibitem{Written:1964xe}
ISLAM KHAN, TARIQ AZIZ, {\em Numerical Solution of a Singularly Perturbed Boundary-Value Problems by Using A Non-Polynomial Spline}.

\bibitem{Written:1964xe}
	Manoj Kumar, P. Singh, Hradyesh Kumar Mishra, {\em An Initial-Value Technique for Singularly Perturbed Boundary Value Problems via Cubic Spline}. {International Journal for Computational Methods in Engineering Science and Mechanics}, {\bf 78 }, 419-427, (02 Oct 2007).

\bibitem{Written:1964xe}
F.Z.Geng, S.P.Qian, S.Li, {\em A numerical method for singularly perturbed turning point problems with an interior layer}. {Journal of Computational and Applied Mathematics}, {\bf 255}, 97-105, (2014).

\bibitem{Written:1964xe}
V. SUBBURAYAN, {\em A parameter uniform numerical method for singularly perturbed delay problems with discontinuous convection coefficient}. {Arab J Math Sci}, {\bf 22}, 191-206, (2016).




\end{thebibliography}
\end{document}